\documentclass[a4paper, oneside]{amsart}
\usepackage{comment}
\usepackage[utf8]{inputenc}
\usepackage{ae,aecompl,amsbsy,amssymb,amsmath,amsthm,mathrsfs,mathtools,eurosym,amsfonts,bm,bbm,esint}
\usepackage[shortlabels]{enumitem}
\usepackage[colorlinks, citecolor = blue]{hyperref}

\usepackage{color}

\makeatletter
\DeclareFontFamily{OMX}{MnSymbolE}{}
\DeclareSymbolFont{MnLargeSymbols}{OMX}{MnSymbolE}{m}{n}
\SetSymbolFont{MnLargeSymbols}{bold}{OMX}{MnSymbolE}{b}{n}
\DeclareFontShape{OMX}{MnSymbolE}{m}{n}{
    <-6>  MnSymbolE5
   <6-7>  MnSymbolE6
   <7-8>  MnSymbolE7
   <8-9>  MnSymbolE8
   <9-10> MnSymbolE9
  <10-12> MnSymbolE10
  <12->   MnSymbolE12
}{}
\DeclareFontShape{OMX}{MnSymbolE}{b}{n}{
    <-6>  MnSymbolE-Bold5
   <6-7>  MnSymbolE-Bold6
   <7-8>  MnSymbolE-Bold7
   <8-9>  MnSymbolE-Bold8
   <9-10> MnSymbolE-Bold9
  <10-12> MnSymbolE-Bold10
  <12->   MnSymbolE-Bold12
}{}

\let\llangle\@undefined
\let\rrangle\@undefined
\DeclareMathDelimiter{\llangle}{\mathopen}%
                     {MnLargeSymbols}{'164}{MnLargeSymbols}{'164}
\DeclareMathDelimiter{\rrangle}{\mathclose}%
                     {MnLargeSymbols}{'171}{MnLargeSymbols}{'171}
\makeatother

\newcommand{\sums}[3]{\sum_{#1 = #2}^{#3}}
\newcommand{\R}{\mathbb{R}}
\newcommand{\C}{\mathbb{C}}

\newcommand{\avgL}{\textit{\L}}

\newcommand{\Z}{\mathbb{Z}}
\newcommand{\Sphere}{\mathbb{S}}

\newcommand{\intD}{\,\mathrm{d}}
\newcommand{\setcolon}{\,\colon\,}
\newcommand{\supp}{\operatorname{supp}}
\newcommand\numberthis{\addtocounter{equation}{1}\tag{\theequation}}

\theoremstyle{plain}
\newtheorem{thm}{Theorem}[section]
\newtheorem{lem}[thm]{Lemma}
\newtheorem{prop}[thm]{Proposition}
\newtheorem{cor}[thm]{Corollary}

\theoremstyle{definition}

\theoremstyle{remark}
\newtheorem{rmrk}[thm]{Remark}

\makeatletter
\@namedef{subjclassname@2020}{%
  \textup{2020} Mathematics Subject Classification}
\makeatother
\address{Aalto University, Department of Mathematics and Systems Analysis, P.O. Box 11100, FI-00076 Aalto, Finland}

\email{aapo.laukkarien@aalto.fi}


\begin{document}

\title[Rough singular integrals]{Convex body domination for Rough singular integrals}
\begin{abstract}
    Convex body domination is a technique, where operators acting on vector-valued functions are estimated via certain convex body averages of the input functions. This domination lets one deduce various matrix weighted bounds for these operators and their commutators.
    In this paper, we extend the sparse domination results for rough singular integrals due to Conde-Alonso, Culiuc, Di Plinio and Ou to the convex body setting. In particular, our methods apply to homogeneous rough singular integrals with unbounded angular part. We also note that convex body domination implies new two weight commutator bounds even in the scalar case.
\end{abstract}

\author{Aapo Laukkarinen}
\maketitle
\setcounter{tocdepth}{1}
\begingroup\hypersetup{linkcolor=black}
\tableofcontents
\endgroup

\section{Introduction}
For many operators acting on scalar functions, the technique of dominating the operator via averages taken over sparse collections of cubes yields sharp weighted estimates for the operator. In the vector case the weights can be matrix-valued, and thus in order to capture information about the direction of the input functions one must replace the averages with convex sets. This domination via convex sets is called convex body domination and it was initiated by Nazarov, Petermichl, Treil and Volberg in \cite{nazarov_convex_2017}. An exciting development appeared very recently in \cite{Domelevo_Matrix_A2_Conj}, where it was shown that the matrix weighted $A_2$ bound in \cite{nazarov_convex_2017} for general Calder\'on-Zygmund operators, which was achieved using convex body domination, is sharp.
Additional convex body domination results can be found, for instance, in 
\cite{AIF_2020__70_5_1871_0}, 
\cite{duong_variation_2021}, \cite{HYTONEN2024127565},  \cite{isralowitz_sharp_2021}, \cite{isralowitz_commutators_2022}, \cite{kakaroumpas2024matrixweighted}, \cite{laukkarinen2023convex}, \cite{muller_quantitative_2022}.

In this paper, we expand the convex body domination theory by showing convex body estimates for some rough singular integrals.  The main examples are rough homogeneous singular integrals
\[
    T_\Omega f(x)=\lim_{\varepsilon\to0}\int_{|y|>\varepsilon} f(x-y)\Omega(\frac{y}{|y|})\frac{\intD y}{|y|^d}, \quad \Omega\in L^q(\Sphere^{d-1}),\quad\int_{\Sphere^{d-1}}\Omega = 0,
\]
and the Bochner-Riesz mean at the critical index that is defined on the frequency side by
\[
    \widehat{B_\delta f}(\xi)=(1-|\xi|^2)_+^\delta\widehat f(\xi), \qquad \delta=\frac{d-1}{2}.
\]
In particular, we show that $T_\Omega$ and $B_\delta$ both satisfy $(L^1,L^p)$ and $(L^p,L^1)$ convex body domination bounds. See Theorem \ref{homCBD}, Corollary \ref{adjhomCBD} and Theorem \ref{CBDBRmeans} for the exact statements. The case $\Omega\in L^\infty(\Sphere^{d-1})$ is a known result due to \cite{muller_quantitative_2022}, but otherwise these bounds are new.

Our results are vector-valued extensions of the results achieved by Conde-Alonso, Culiuc, Di Plinio and Ou in \cite{10.2140/apde.2017.10.1255}, and our method of proof is an adaptation of their techniques. We will study reasonably general sesquilinear forms $\Lambda_\mu^\nu$ that admit a multi-scale structure
\[
    \Lambda_\mu^\nu(f,g)=\sum_{\mu<s\leq\nu} \int K_s(x,y) f(y)\bar g(x)\intD(y,x),
\]
and prove a convex body domination principle for $\sup_{\mu<\nu}|\Lambda_\mu^\nu(f,g)|$. See Theorem \ref{main} for more details. The assumptions on the sesquilinear form are scalar valued so the fact that $T_\Omega$ and $B_\delta$ fall under the umbrella of this general principle follows directly from the considerations in \cite{10.2140/apde.2017.10.1255}. 

In the case of $\Omega\in L^\infty(\Sphere^{d-1})$, similar bounds have been discovered also in \cite{AIF_2020__70_5_1871_0}, where they even estimate the larger form
\[  
    \int \sup_{\mu<\nu}|\sum_{\mu<s\leq\nu}\int K_s(x,y)f(y)\intD y|\,|g(x)|\intD x
\]
and get a slightly weaker $(L^p,L^p)$ bound. Convex body domination bounds for multi-scale operators have also been done before by the author in \cite{laukkarinen2023convex}. The  boundedness assumptions in \cite{laukkarinen2023convex} are weaker than we have here, but on the other hand, in the setting  of this paper we do not assume any regularity, which is present in \cite{laukkarinen2023convex}.

Due to results in \cite{kakaroumpas2024matrixweighted} and \cite{laukkarinen2023convex}, the convex body bounds for $T_\Omega$ and $B_\delta$ imply various matrix-weighted bounds for these operators and their commutators. We will also deduce new scalar-valued results from the convex body domination. In \cite{laukkarinen2023convex} it was shown that convex body domination implies a certain sparse domination result for the commutators. A result in \cite{Lerner2024bloom} says that this sparse domination of commutators can used to deduce two weight bounds in the so-called Bloom-type setting, and these Bloom-type bounds are completely new for $T_\Omega$ in the case where $\Omega$ is unbounded. In a forthcoming work with J. Sinko \cite{laukkarinencompactness}, we obtain further applications of these results to the compactness of commutators. 

The rest of this paper is organized as follows. In Section 2, we follow the set up from \cite{10.2140/apde.2017.10.1255} by introducing the sesquilinear forms in detail and define collections of cubes called stopping collections that play an essential role in the proof of the general domination principle. In Section 3, following \cite{HYTONEN2024127565} and \cite{kakaroumpas2024matrixweighted}, we gather the definitions that are needed to perform convex body domination in a general complex Banach space. Section 4 is dedicated to the proof of the general convex body domination principle, which is an adaptation of the general sparse domination argument from \cite{10.2140/apde.2017.10.1255}.  In Section 5, we outline the argument from \cite{10.2140/apde.2017.10.1255}, which shows that the general principle can be applied to $T_\Omega$. In Section 6 we gather the matrix-weighted estimates that follow from the convex body domination.  Another application is explored in Section 7, where we show scalar-valued two weight bounds for $[b,T_\Omega]$ in the Bloom-type setting. In Section 8, we show that the same results hold for $B_\delta$, but we present a different proof that uses weak-type control of the so-called grand maximal truncation operator. 

\section*{Acknowledgements}
The author would like to thank Tuomas Hytönen for valuable discussions and comments that improved the presentation. The author is a member of the Finnish Centre of Excellence in Randomness and Structures supported by the Research Council of Finland through grants 346314 and 364208.
\section*{Notation}
\begin{tabular}{p{1.5cm}p{10cm}}
$d,n$& Dimensional constants, these are always positive integers.\\[1ex]
$p'$& Hölder conjugate of a number $p\in[1,\infty]$, $p'=\frac{p}{p-1}$.\\[1ex]
$A\lesssim_{P}B$& Inequality with implicit constant that depends only on parameters $P$, i.e., $A\leq C_PB$.\\ [1ex]
$\mathbf{sh}\mathcal{Q}$& Union over a collection $\mathcal{Q}$ of sets, $\bigcup_{Q\in\mathcal{Q}}Q$.\\[1ex]
$|A|$& The $d$-dimensional Lebesque measure of a set $A$.\\[1ex]
$|x|$& The Euclidean norm of a vector $x$.\\[1ex]
$\bar x$& Complex conjugate of an element $x$ in a complex vector space.\\[1ex]
$\bar{B}_X$&Closed unit ball of a normed space $X$.\\[1ex]
$T^*$& The adjoint of an operator $T$. In particular, if $T$ is a complex matrix, then $T^*$ is the conjugate transpose of $T$.\\[1ex]
$\langle f, g\rangle$ & The usual integral pairing $\int  f(x)\bar  g(x)\intD x$, where $f, g$ are $\C$-valued. \\[1ex]
$\|f\|_{\avgL^p(Q)}$ & The $L^p$ average over a set $Q$, i.e., $(\fint_Q|f|^p)^\frac{1}{p}$. \\ [1ex]
$\avgL^p(Q)$& The normed space of functions $f$ satisfying $\|f\|_{\avgL^p(Q)}<\infty$.\\[1ex]
$M_p$& The maximal operator $M_pf(x)=\sup_{Q\ni x}\|f\|_{\avgL^p(Q)}$, here $Q$ is a cube and $p\in[1,\infty)$.\\[1ex]$M_p^\mathscr D$& The dyadic maximal operator $M_p^\mathscr Df(x)=\sup_{Q\ni x}\|f\|_{\avgL^p(Q)}$, here $Q$ is a dyadic cube and $p\in[1,\infty)$.\\[1ex]
$\llangle \vec f\rrangle_{\avgL^p(Q)}$& The $L^p$ convex body of $\vec f=(f_1,\dots f_n)$ related to a set $Q$, i.e., the set $\{(|Q|^{-1}\langle \mathbbm 1_Qf_i,\varphi\rangle)_{i=1}^n\setcolon \varphi\in\bar{B}_{\avgL^{p'}(Q)}\}$.
\end{tabular}

\section{Sesquilinear forms and stopping collections}
We follow the setting from \cite{10.2140/apde.2017.10.1255}. Let $1<r<\infty$ and $\Lambda$ be an $L^r(\R^d)\times L^{r'}(\R^d)$-bounded sesquilinear form whose kernel $K=K(x,y)$ coincides with a function away from the diagonal. More
precisely, whenever $f\in L^r(\R^d)$ and $g\in L^{r'}(\R^d)$ are compactly and disjointly supported
\[
    \Lambda(f,g)=\int_{\R^d}\int_{\R^d}K(x,y)f(y)\intD y\, \bar g(x)\intD x
\]
with absolute convergence of the integral.
\begin{rmrk}
    In contrast to \cite{10.2140/apde.2017.10.1255} we define $\Lambda$ as a sesquilinear (instead of bilinear) form in order to have a basis-independent extension for vector-valued functions (see Section \ref{compConvBodySec}). This rarely plays a role in calculations.
\end{rmrk}
We assume that
the kernel $K$ of $\Lambda$ admits the decomposition
\begin{align*}
    K(x,y)=\sum_{s\in\Z}K_s(x,y),
\end{align*}
where each piece is supported in an annulus away from the diagonal, i.e.,
\begin{equation}
    \supp K_s\subset \{(x,y)\in\R^d\times\R^d\setcolon x-y\in A_s\},\tag{S}\label{supportCond}
\end{equation}
where
\[
     A_s\coloneqq \{z\in\R^d\setcolon 2^{s-2}<|z|<2^s\}.
\]
Furthermore, we assume that there exists $1 < q \leq \infty$ such that
\[
    [K]_q\coloneqq\sup_{s\in\Z}2^\frac{sd}{q'}\sup_{x\in\R^d}(\|K_s(x,\cdot)\|_q+\|K_s(\cdot,x)\|_q)<\infty,\tag{K}\label{KernelCond}
\]
and the associated truncated forms
\[
    \Lambda_\mu^\nu(h_1,h_2)\coloneqq\int_{\R^d\times\R^d}\sum_{\mu<s\leq\nu} K_s(x,y)h_1(y)\bar h_2(x)\intD y\intD x
\]
satisfy
\begin{equation}\label{bddCond}
    C_T(r)\coloneqq \sup_{\mu<\nu}\|\Lambda_\mu^\nu\|_{L^r(\R^d)\times L^{r'}(\R^d)\to\C}<\infty.\tag{T}
\end{equation}

Let $\mathscr{D}$ be a dyadic lattice in $\R^d$. We fix a dyadic cube $Q\in\mathscr{D}$ and denote $s_Q\coloneqq \log_2 \ell(Q)$. A collection $\mathcal{Q}\subset\mathscr{D}$ of dyadic cubes is called a stopping collection with top $Q$ if the elements of $\mathcal Q$ are pairwise disjoint, contained in $3Q$ and enjoy the separation properties
\begin{equation}
    L,L'\in\mathcal Q,\, |s_L-s_{L'}|\geq 8 \quad\Rightarrow \quad7L\cap7L'=\emptyset
\end{equation}
and 
\begin{equation}
    \bigcup_{\substack{L\in\mathcal Q\\3L\cap2Q\neq\emptyset}}9L\subset\bigcup_{L\in\mathcal Q}L\eqqcolon \mathbf{sh}\mathcal Q.
\end{equation}

We also define the local sesquilinear forms by
\begin{equation*}
    \Lambda_{Q,\mu,\nu}(h_1,h_2) \coloneqq \Lambda_{\mu}^{\min\{s_Q,\nu\}}(h_1\mathbbm 1_Q,h_2)=\Lambda_{\mu}^{\min\{s_Q,\nu\}}(h_1\mathbbm 1_Q,h_2\mathbbm 1_{3Q})
\end{equation*}
for all dyadic cubes $Q$. The last equality is true due to the support condition \eqref{supportCond}. Furthermore, given a stopping collection $\mathcal Q$ with top $Q$, we define 
\[
    \Lambda_{\mathcal Q,\mu,\nu}(h_1,h_2)\coloneqq \Lambda_{Q,\mu,\nu}(h_1,h_2)-\sum_{\substack{L\in\mathcal Q\\L\subset Q}} \Lambda_{L,\mu,\nu}(h_1,h_2)= \Lambda_{\mathcal Q,\mu,\nu}(h_1\mathbbm 1_Q,h_2\mathbbm 1_{3Q}).
\]
Again, the last equality is due to \eqref{supportCond}. Note that $\Lambda_{\mathcal Q,\mu,\nu}=\Lambda_{\mathcal Q',\mu,\nu}$, where $\mathcal{Q}'\coloneqq \{L\in\mathcal{Q}\setcolon L\subset Q\}$.

For $1 \leq p \leq \infty$, define $\mathcal{Y}_p(\mathcal Q)$ to be the  subspace of $L^p(\R^d)$ of functions $h$ that are supported in $3Q$ and the quantity
\[
    \|h\|_{\mathcal{Y}_p(\mathcal Q)}\coloneqq\begin{cases}
        \max \left\{\|h\mathbbm 1_{\R^d\setminus \mathbf{sh}\mathcal Q}\|_\infty,\,\displaystyle{\sup_{L\in\mathcal Q}\inf_{x\in\widehat{L}}}M_ph(x)\right\},\quad &p<\infty\\
        \|h\|_\infty, &p=\infty,
    \end{cases}
\]
 where $\widehat L$ is the $2^5$-fold dilate of $L$, is finite. We denote $\mathcal{X}_p(\mathcal Q)$ the subspace of $\mathcal{Y}_p(\mathcal Q)$ of functions satisfying 
\[
    b=\sum_{L\in\mathcal Q} b_L,\quad \supp b_L\subset L. 
\]
We note that the above condition is equivalent to $\supp b\subset \mathbf{sh}\mathcal Q$.
We also write $b\in\dot{\mathcal{X}}_p(\mathcal Q)$ if $b\in \mathcal{X}_p(\mathcal Q)$ and 
\[
    \int_L b_L = 0 \quad \text{ for all } L\in\mathcal Q.
\]
We make a standing assumption that for some $p_1,p_2\in[1,\infty)$ there exists a constant $C_L$ such that 
\begin{align}
    |\Lambda_{\mathcal Q, \mu,\nu}(b,h)|\leq C_L \,|Q|\,\|b\|_{\dot{\mathcal{X}}_{p_1}(\mathcal Q)}\|h\|_{\mathcal{Y}_{p_2}(\mathcal Q)},\label{localBddCond1}\tag{L1}\\|\Lambda_{\mathcal Q, \mu,\nu}(h,b)|\leq C_L \,|Q|\,\|h\|_{\mathcal{Y}_{\infty}(\mathcal Q)}\|b\|_{\dot{\mathcal{X}}_{p_2}(\mathcal Q)}\label{localBddCond2}\tag{L2}
  \end{align}
hold uniformly over all $\mu, \nu \in \Z$, all dyadic lattices $\mathscr{D}$, all $Q \in \mathscr D$ and all stopping collections
$\mathcal Q \subset \mathscr D$ with top $Q$. Note that conditions \eqref{localBddCond1} and \eqref{localBddCond2} are not symmetric.

In order to verify \eqref{localBddCond1} and \eqref{localBddCond2} in concrete situations, the following representations of the sesquilinear forms $\Lambda_{\mathcal{Q},\mu,\nu}$ will turn out to be helpful. These representations are sesquilinear form versions of \cite[Equations (2-9)--(2-11)]{10.2140/apde.2017.10.1255}, and we will formulate them in the following lemma.
\begin{lem}\label{representations}
    Assume that $K$ satisfies \eqref{supportCond} and \eqref{KernelCond} with $1<q\leq\infty$. If $b\in\mathcal{X}_1(\mathcal{Q})$ with $\supp b\subset Q$, then
\[
    \Lambda_{\mathcal{Q},\mu,\nu}(b,h)=\sum_{j\geq1}\int \sum_{\mu<s\leq\min\{s_Q,\nu\}}K_s(x,y)b_{s-j}(y)\bar h(x)\intD y\intD x,\numberthis\label{rep1}
\]
where \[b_s\coloneqq\sum_{\substack{L\in\mathcal{Q}\\s_L=s}}b_L.\] 
If $h\in \mathcal{Y}_{q'}(\mathcal{Q})$ with $\supp h\subset Q$, and $b\in \mathcal{X}_{q'}(\mathcal{Q})$, then 
\[
    \Lambda_{\mathcal{Q},\mu,\nu}(h,b)=\overline{\Lambda^*_{\mathcal{Q},\mu,\nu}(b^\text{in},h)} + V_\mathcal{Q}(h, b), \numberthis\label{rep2}
\]
where $\Lambda^*$ is like $\Lambda$, but with $\overline{K(y,x)}$ in place of $K(x,y)$, the function $b^\text{in}$ is a certain truncation of $b$ that satisfies $\|b^\text{in}\|_{\mathcal{X}^{p}(\mathcal{Q})}\leq \|b\|_{\mathcal{X}^{p}(\mathcal{Q})}$ for any $1\leq p\leq\infty$
and the remainder term $V_\mathcal{Q}$ satisfies
\[
    |V_\mathcal{Q}(h, b)|\leq2^{\vartheta d}[K]_q|Q|\|h\|_{\mathcal{Y}_{q'}(\mathcal{Q})}\|b\|_{\mathcal{X}_{q'}(\mathcal{Q})}\numberthis\label{rep2rem}
\]
for a suitable absolute constant $\vartheta$.    
\end{lem}
\begin{proof}
We will first deal with \eqref{rep1}. Note that for $b\in\mathcal{X}_1(\mathcal{Q})$ with $\supp b \subset Q$ we have
\begin{align*}
    \Lambda_{\mathcal{Q},\nu,\mu}(b,h)&=\Lambda_{Q,\nu,\mu}(b,h)-\sum_{\substack{L\in\mathcal{Q}\\L\subset Q}}\Lambda_{L,\nu,\mu}(b,h)\\&=\Lambda_{Q,\nu,\mu}\Big(\sum_{L\in\mathcal{Q}}b_L,h\Big)-\sum_{L\in\mathcal{Q}}\Lambda_{L,\nu,\mu}(b_L,h).
\end{align*}
Since $\mathcal{Q}$ is a disjoint collection, we may bring the first sum outside the integral, which yields
\begin{align*}
    \Lambda_{\mathcal{Q},\nu,\mu}(b,h)&=\sum_{L\in\mathcal{Q}}\Bigg(\int \sum_{\mu<s\leq\min\{s_Q,\nu\}}K_s(x,y)b_{L}(y)\bar h(x)\intD y\intD x\\&\qquad\qquad\qquad\qquad- \int\sum_{\mu<s\leq\min\{s_L,\nu\}}K_s(x,y)b_{L}(y)\bar h(x)\intD y\intD x\Bigg)\\
    &=\sum_{L\in\mathcal{Q}}\int \sum_{\substack{\mu<s\leq\min\{s_Q,\nu\}\\s>s_L}}K_s(x,y)b_{L}(y)\bar h(x)\intD y\intD x\\
    &=\int \sum_{\mu<s\leq\min\{s_Q,\nu\}}K_s(x,y)\sum_{\substack{L\in\mathcal{Q}\\s_L<s}}b_{L}(y)\bar h(x)\intD y\intD x.
\end{align*}
Now the wanted representation \eqref{rep1} follows since \[\sum_{\substack{L\in\mathcal{Q}\\s_L<s}}b_{L}=\sum_{j\geq1}\sum_{\substack{L\in\mathcal{Q}\\s_L=s-j
    }}b_{L}=\sum_{j\geq1}b_{s-j},\]
and the functions $b_s$ are disjointly supported so we can bring the sum outside the integral. 

The second representation \eqref{rep2}--\eqref{rep2rem} was stated in \cite[Remark 2.6]{10.2140/apde.2017.10.1255} and proved in \cite[Appendix A]{10.2140/apde.2017.10.1255} for the bilinear forms $\Lambda^{\text{bi}}(b,h)\coloneqq\Lambda(b,\bar h)$. More precisely, they showed that
\[
    \Lambda^{\text{bi}}_{\mathcal{Q},\mu,\nu}(h,b)=\tilde\Lambda^{\text{bi}}_{\mathcal{Q},\mu,\nu}(b^{\text{in}},h)+V^{\text{bi}}_\mathcal{Q}(h,b),
\]
where $\tilde\Lambda^{\text{bi}}$ is like $\Lambda^{\text{bi}}$, but with $K(y,x)$ in place of $K(x,y)$, the function $b^{\text{in}}$ is as in the claim of this Lemma, and the remainder term satisfies 
\[
    |V_\mathcal{Q}^{\text{bi}}(h,b)| \leq2^{\vartheta d}[K]_q|Q|\|h\|_{\mathcal{Y}_{q'}(\mathcal{Q})}\|b\|_{\mathcal{X}_{q'}(\mathcal{Q})}\numberthis\label{bilinRepRem}.
\] 
Applying this result and setting $V_\mathcal{Q}(h,b)\coloneqq V^\text{bi}_\mathcal{Q}(h,\bar b)$ we get
\[
    \Lambda_{\mathcal{Q},\mu,\nu}(h,b)=\Lambda^{\text{bi}}_{\mathcal{Q},\mu,\nu}(h,\bar b)= \tilde \Lambda^{\text{bi}}_{\mathcal{Q},\mu,\nu}((\bar{b})^{\text{in}},h) +V_\mathcal{Q}(h,b).
\]
We observe from \cite[page 1263]{10.2140/apde.2017.10.1255}  that the function $b^\text{in}$ satisfies $(\bar{b})^{\text{in}}=\overline{b^{\text{in}}}$. 
Then using \eqref{rep1} it is simple to check that 
\[
    \tilde \Lambda^{\text{bi}}_{\mathcal{Q},\mu,\nu}(\overline{b^{\text{in}}},h)= \overline{\Lambda^*_{\mathcal{Q},\mu,\nu}(b^{\text{in}},h)}.
\]
This proves \eqref{rep2}, and
 substituting $\bar b$ for $b$ in \eqref{bilinRepRem} we see that \eqref{rep2rem} also holds.
\end{proof}

\section{Complex convex bodies}\label{compConvBodySec}
In this Section we gather the necessary definitions and results from \cite{HYTONEN2024127565} that are needed to perform convex body domination. In particular, we will need complex normed space versions of these results. We do not claim any originality for the results or the setting in this Section, and even the translation of language from real to complex has been done before in more detail, see \cite[Appendix A]{DOMELEVO2024127956} and \cite[Section 3.1]{kakaroumpas2024matrixweighted}.

For the entirety of this Section we assume that $X,Y$ are complex normed spaces and $\vec x\coloneqq(x_i)_{i=1}^n,\vec y\coloneqq(y_i)_{i=1}^n$ are elements of $X^n,Y^n$ respectively. First we extend the standard Hermitian dot product of $\C^n$ in a natural way to $X^n\times\C^n,Y^n\times\C^n$. In other words, if $Z$ is a complex normed space, $\vec z \in Z^n$ and $\vec v=(v_i)_{i=1}^n\in\C^n$, then we write
\[
    \vec z\cdot \vec v \coloneqq \sum_{i=1}^n z_i\bar v_i.
\]
Then we extend a sesquilinear form $t\,\colon X\times Y\to \C$ to $X^n\times Y^n$ by
\[
    t(\vec x, \vec y)\coloneqq \sum_{i=1}^n t(\vec x\cdot \vec e_i, \vec y\cdot\vec e_i),
\]
where $\{\vec e_i\}_{i=1}^n$ is an orthonormal basis of $\C^n$.
It is important to note that we have used a sesquilinear form instead of a bilinear form in order to have an extension that is independent of the choice of basis. Furthermore, if $A\in \C^{n\times n}$ is a linear transformation, then 
\[
    t(A\vec x,\vec y)=t(\vec x,A^*\vec y).
\]
The convex bodies are defined by 
\[
    \llangle \vec x\rrangle_X\coloneqq \{\langle \vec x,x^*\rangle_{(X,X^*)}\setcolon x^*\in \Bar B_{X^*}\}\subset \C^n,
\]
where $X^*$ is the dual space of $X$ and $\langle \vec x,x^*\rangle_{(X,X^*)}\coloneqq(\langle x_i,x^*\rangle_{(X,X^*)})_{i=1}^n$. 
We say that a set $K\in\C^n$ is complex-symmetric around the origin, if for any $\vec k\in K$ and $\lambda\in\C$ with $|\lambda|=1$ it holds that $\lambda\vec k\in K$.
The set $\llangle \vec x\rrangle_X$ is compact, convex and complex-symmetric around the origin, and it follows that
the Minkovski dot product
\[
    \llangle \Vec{x}\rrangle_X\cdot\llangle \Vec{y}\rrangle_Y=\{a\cdot b\setcolon a\in\llangle \Vec{x}\rrangle_X,\,b\in\llangle \Vec{y}\rrangle_Y\}\subset \C
\]
is also compact, convex and complex-symmetric around the origin. This means that for some number $c\in [0,\infty)$ the above set is equal to a closed disk with radius $c$, and we may identify the whole set with $c$.

We end this Section with the main tool that we will use to transition from norm bounds to convex body bounds. The following is the complex normed space version of  \cite[Lemma 4.1]{HYTONEN2024127565}. For a proof in the complex case see \cite[Lemma C]{kakaroumpas2024matrixweightedold}.
\begin{lem}\label{TH1}
    Let $\mathcal{E}_x$ be the John ellipsoid of $\llangle \Vec{x}\rrangle_X$ such that 
    \[
        \mathcal{E}_x\subset \llangle \Vec{x}\rrangle_X\subset \sqrt{n}\mathcal{E}_x,
    \]
    and suppose that $\mathcal{E}_x$ is non-degenerate (i.e. of full dimension). Let $R_x$ be a linear transformation such that $R_x\mathcal{E}_x=\Bar{B}_{\mathbb{C}^n}$ and let $(\Vec{e}_i)_{i=1}^n$ be an orthonormal basis of $\mathbb{C}^n$. If 
    \[x_i\coloneqq R_x\Vec{x}\cdot\Vec{e}_i,\quad y_i\coloneqq ({R_x}^*)^{-1}\Vec{y}\cdot\Vec{e}_i,\quad i=1,\dots,n,\]
    then
    \[
        \sums{i}{1}{n}\|x_i\|_X\|y_i\|_Y\leq n^\frac{3}{2}\llangle \Vec{x}\rrangle_X\cdot\llangle \Vec{y}\rrangle_Y.
    \]
\end{lem}

\section{An abstract convex body domination principle}
For this section we will suppose that $\Lambda\colon L^r(\R^d)\times L^{r'}(\R^d)\to\C$ is a bounded sesquilinear form that satisfies  \eqref{supportCond}, \eqref{bddCond}, \eqref{localBddCond1} and \eqref{localBddCond2} with parameters $p_1,p_2\in[1,\infty)$. We note that the kernel size condition \eqref{KernelCond} will not be explicitly needed for the convex body domination principle proven in this section. However, in many applications, Lemma \ref{representations} is essential for checking the conditions \eqref{localBddCond1} and \eqref{localBddCond2}, and Lemma \ref{representations} does assume \eqref{KernelCond}.

We will now strive to show a convex body domination principle for $\Lambda$. The proof is based on techniques in \cite{10.2140/apde.2017.10.1255} coupled with the convex body domination theory from \cite{HYTONEN2024127565} that we discussed in the previous section. The following constant will appear throughout this section. For some absolute constant $\Theta$ we define
\begin{equation}\label{sparseConst}
    \mathcal{C}_{\Theta,d,n}\coloneqq 2^{\Theta d}[C_T(r)+C_L]n^{\frac{3}{2}+\frac{1}{p_1}+\frac{1}{p_2}}.
\end{equation}
The following result is contained in the proof of \cite[Lemma 2.7]{10.2140/apde.2017.10.1255}.
\begin{lem}\label{softLocalBdd}
    Let $Q$ be a dyadic cube in $\R^d$, $1\leq p_1,p_2<\infty$ and $\mathcal Q$ be a stopping collection with top $Q$.
    Then there exists an absolute constant $\vartheta$ such that
    \[
        |\Lambda_{\mathcal Q,\mu,\nu}(h_1,h_2)|\leq 2^{\vartheta d}\mathbf{C}|Q|\|h_1\|_{\mathcal{Y}_{p_1}(\mathcal Q)}\|h_2\|_{\mathcal{Y}_{p_2}(\mathcal Q)},
    \]
    where $\mathbf{C}\coloneqq C_T(r)+C_L$.
\end{lem}

We will now prove the following Proposition, which is a convex body version of \cite[Equation (2-20)]{10.2140/apde.2017.10.1255}.
\begin{prop}\label{inductionBase}
Let $Q_0$ be a dyadic cube in $\R^d$ and $1\leq p_1,p_2<\infty$. Then there exists  absolute constants $\vartheta, \Theta$ and a stopping collection $\mathcal S_1$ with top $Q_0$ such that
\begin{equation}\label{StopColCubeSum}
    \sum_{L\in \mathcal S_1}|L|\leq 2^{-\vartheta d}|Q_0|
\end{equation}
and
\begin{align}\label{BasecaseEst}
    |\Lambda_{Q_0,\mu,\nu}(\vec f, \vec g)|\leq \mathcal{C}_{\Theta,d,n}|Q_0|\llangle\vec f\rrangle_{\avgL^{p_1}(3Q_0)}\cdot& \llangle\vec g\rrangle_{\avgL^{p_2}(3Q_0)}+\sum_{\substack{L\in\mathcal S_1\\L\subset Q_0}}|\Lambda_{L,\mu,\nu}(\vec f,\vec g)|,
\end{align}
where $\mathcal{C}_{\Theta,d,n}$ is as in \eqref{sparseConst}.
\end{prop}
\begin{proof}
We begin the proof by noticing that we may assume that the John ellipsoid of $\llangle \vec f\rrangle_{\avgL^{p_1}(3Q_0)}$ is non-degenerate. Indeed, let $P$ denote the orthogonal projection of $\C^n$ onto the linear span of $\llangle f\rrangle_{\textit{\L}^p(3Q_0)}$ that we denote by $H$. Note that now we have $\vec f = P\vec f=P^*P\vec f$ almost everywhere, which means that on the left-hand side of \eqref{BasecaseEst} we have 
\[
    \Lambda_{Q_0,\mu,\nu}(\vec f, \vec g)=\Lambda_{Q_0,\mu,\nu}(P^* P\vec f, \vec g)=\Lambda_{Q_0,\mu,\nu}(P\vec f, P\vec g),
\]
and on the right-hand side the same identity for $\Lambda_{L,\mu,\nu}$ and
\begin{align*}
    \llangle P\Vec{f}\rrangle_{\textit{\L}^{p_1}(3Q_0)}\cdot \llangle P\Vec{g}\rrangle_{\textit{\L}^{p_2}(3Q_0)}&=\llangle P^* P\Vec{f}\rrangle_{\textit{\L}^{p_1}(3Q_0)}\cdot \llangle\Vec{g}\rrangle_{\textit{\L}^{p_2}(3Q_0)}\\&=\llangle \Vec{f}\rrangle_{\textit{\L}^{p_1}(3Q_0)}\cdot \llangle\Vec{g}\rrangle_{\textit{\L}^{p_2}(3Q_0)}.
\end{align*}
Thus if suffices to prove the statement with $\vec f,\vec g$ taking values in $H$, and the John ellipsoid of $\llangle f\rrangle_{\textit{\L}^p(3Q_0)}$ is non-degenerate as a subset of $H$.

For each cube $Q$, we will consider the linear transformation  $R_Q$ that maps the John ellipsoid of the convex body $\llangle\vec f\rrangle_{\avgL^{p_1}(3Q)}$ to the unit ball of $\C^n$. We fix an orthonormal basis $(\vec e_i)_{i=1}^n$ of $\C^n$, and
write $f_{i}^Q\coloneqq R_Q\vec f\cdot \vec e_i$ and $g_{i}^Q\coloneqq ({R_Q}^*)^{-1}\vec g\cdot \vec e_i$.

We also define the sets
\[
    E_{Q,f,i}\coloneqq\{x\in3Q\setcolon\frac{M_{p_1}(f_i^Q\mathbbm 1_{3Q})(x)}{\|f^Q_i\|_{\avgL^{p_1}(3Q)}}>(2^\frac{\Theta d}{4}n)^\frac{1}{p_1}\},
\]
\[
    E_{Q,g,i}\coloneqq\{x\in3Q\setcolon\frac{M_{p_2}(g_i^Q\mathbbm 1_{3Q})(x)}{\|g_i^Q\|_{\avgL^{p_2}(3Q)}}>(2^\frac{\Theta d}{4}n)^\frac{1}{p_2}\},
\]
\[E_{Q,i}\coloneqq E_{Q,f,i}\cup E_{Q,g,i}\quad \text{ and }\quad E_Q\coloneqq\bigcup_{i=1}^nE_{Q,i}.\]
To simplify notation we will write $f_i\coloneqq f_i^{Q_0}$ and $g_i\coloneqq g_i^{Q_0}$ during this proof. Notice that due to the weak-type boundedness of the maximal function, provided that $\Theta$ is sufficiently large, we have \[|E_Q|\leq 2^{-\vartheta d}|Q|,\numberthis\label{majorEst}\]
where $\vartheta$ is the absolute constant from Lemma \ref{softLocalBdd}. 

Then we set $E_0\coloneqq{3Q_0}$, $\mathcal{S}_0\coloneqq\{Q_0\}$,
\[E_1\coloneqq E_{Q_0}\quad\text{ and }\quad\mathcal{S}_1\coloneqq\text{maximal cubes $L\in \mathscr{D}$ such that $9L\subset E_1$}.\]
Now $\mathcal{S}_1$ is a pairwise disjoint collection, and the following properties are satisfied: 
\begin{align}
    &E_1=\bigcup_{L\in\mathcal{S}_1}L= \bigcup_{L\in\mathcal{S}_1}9L\subset E_0,\label{nested1}\\
    &|Q_0\setminus E_1|\geq (1-2^{-d\vartheta})|Q_0|,\label{sparseness1}\\
    &L,L'\in\mathcal{S}_1,\, 7L\cap7L'\neq\emptyset \quad\Rightarrow\quad |s_L-s_{L'}|<8. \label{neigbour1}
\end{align}
Indeed, disjointness and \eqref{nested1} are true by construction, and \eqref{sparseness1} follows from \eqref{majorEst}. For \eqref{neigbour1} suppose instead that $7L\cap7L'$
is not empty for some
$s_L \leq s_{L'}-8$. We recall that $\widehat L$ is the $2^5$-fold dilate of $L$. It is not hard to check that $\widehat L \subset 9L'$, which implies
that the $9$-fold dilate of the dyadic parent of $L$ is contained in $9L'$
as well, contradicting the
maximality of $L$. 

From these properties it follows that $\mathcal{S}_1$ is a stopping collection with top $Q_0$, and the rest of the proof is dedicated to checking that this collection satisfies \eqref{StopColCubeSum} and \eqref{BasecaseEst}.

Due to \eqref{nested1} and \eqref{majorEst}, we have
\[
    \sum_{L\in \mathcal S_1} |L|=|E_1|=|E_{Q_0}|\leq 2^{-\vartheta d}|Q_0|,
\]
which is exactly \eqref{StopColCubeSum}.

In order to check \eqref{BasecaseEst} we apply Lemma \ref{softLocalBdd} to get
\begin{align*}
     |\Lambda_{Q_0,\mu,\nu}(\vec f,\vec g)|&=\Big|\Lambda_{\mathcal S_1,\mu,\nu}(\vec f,\vec g)+\sum_{\substack{L\in\mathcal S_1\\L\subset Q_0}}\Lambda_{L,\mu,\nu}(\vec f,\vec g)\Big|\\&\leq\sum_{i=1}^n|\Lambda_{\mathcal S_1,\mu,\nu}(f_i,g_i)|+\sum_{\substack{L\in\mathcal S_1\\L\subset Q_0}}|\Lambda_{L,\mu,\nu}(\vec f,\vec g)|\\&\leq2^{\vartheta d}\mathbf{C}|Q_0| \sum_{i=1}^n\|f_i\|_{\mathcal{Y}_{p_1}(\mathcal S_1)}\|g_i\|_{\mathcal{Y}_{p_2}(\mathcal S_1)}\\&\qquad\qquad\qquad\qquad\qquad\qquad\qquad+\sum_{\substack{L\in\mathcal S_1\\L\subset Q_0}}|\Lambda_{L,\mu,\nu}(\vec f,\vec g)|,
\end{align*}
where $\mathbf{C}\coloneqq C_T(r)+C_L$. By \eqref{nested1} we have that $\mathbf{sh} \mathcal{S}_1=\bigcup_{L\in\mathcal{S}_1}L=E_1=E_{Q_0}$ so we get
\[
    \sup_{x\notin \mathbf{sh}\mathcal S_1}|f_i(x)\mathbbm 1_{3Q_0}(x)|\leq (2^\frac{\Theta d}{4}n)^\frac{1}{p_1}\|f_i\|_{\avgL^{p_1}(3Q_0)}, \quad i=1,\dots,n.
\] 
From the maximality of $\mathcal{S}_1$ we get that if $L\in \mathcal{S}_1$, then the dyadic parent $P(L)$ of $L$ satisfies $9P(L)\not\subset E_1=E_{Q_0}$, and hence 
\begin{align*}
    \sup_{L\in\mathcal S_1}\inf_{x\in\widehat L} M_{p_1}(f_i\mathbbm 1_{3Q_0})(x)&\leq \sup_{L\in\mathcal S_1}\inf_{x\in 9P(L)} M_{p_1}(f_i\mathbbm 1_{3Q_0})(x)\\&\leq (2^\frac{\Theta d}{4}n)^\frac{1}{p_1}\|f_i\|_{\avgL^{p_1}(3Q_0)},\qquad \quad i=1,\dots,n.
\end{align*}
Similarly, the above estimates hold with $(g_i,p_2)$ instead of $(f_i,p_1)$. Since $p_1,p_2\geq1$, these estimates imply that
\[
    \|f_i\|_{\mathcal{Y}_{p_1}(\mathcal S_1)}\|g_i\|_{\mathcal{Y}_{p_2}(\mathcal S_1)} \leq 2^\frac{\Theta d}{2}n^{\frac{1}{p_1}+\frac{1}{p_2}}\|f_i\|_{\avgL^{p_1}(3Q_0)}\|g_i\|_{\avgL^{p_1}(3Q_0)}.
\]
Now we apply Lemma \ref{TH1} to conclude
\begin{align*}
    \sum_{i=1}^n\|f_i\|_{\mathcal{Y}_{p_1}(\mathcal S_1)}\|g_i\|_{\mathcal{Y}_{p_2}(\mathcal S_1)}&\leq 2^\frac{\Theta d}{2}n^{\frac{1}{p_1}+\frac{1}{p_2}}\sum_{i=1}^n\|f_i\|_{\avgL^{p_1}(3Q_0)}\|g_i\|_{\avgL^{p_1}(3Q_0)}\\
    &\leq 2^\frac{\Theta d}{2} n^{\frac{3}{2}+\frac{1}{p_1}+\frac{1}{p_2}}\llangle\vec f\rrangle_{\avgL^{p_1}(3Q_0)}\cdot \llangle\vec g\rrangle_{\avgL^{p_2}(3Q_0)}.
\end{align*}
\end{proof}
The above Proposition provides the base of the induction that paves the way for the abstract convex body domination principle that we formulate in the following Theorem. It is the convex body counterpart of \cite[Theorem C]{10.2140/apde.2017.10.1255}.
\begin{thm}\label{main}
    Let $1\leq p_1,p_2<\infty$, and let $\Lambda$ be a sesquilinear form satisfying \eqref{supportCond},  \eqref{bddCond}, \eqref{localBddCond1} and \eqref{localBddCond2}. Then there exists a positive absolute constant $\Theta$ such that $\Lambda$ satisfies convex body domination. More explicitly, for all $\vec f\in L^{p_1}(\R^d,\C^n)$ and $\vec g\in L^{p_2}(\R^d,\C^n)$ with compact support, there exists a sparse collection $\mathcal{S}$ of cubes in $\R^d$ such that we have 
    \[
        \sup_{\mu<\nu}|\Lambda_\mu^\nu(\vec f, \vec g)|\leq \mathcal{C}_{\Theta,d,n}\sum_{Q\in\mathcal{S}}|Q|\llangle\vec f\rrangle_{\avgL^{p_1}(Q)}\cdot \llangle\vec g\rrangle_{\avgL^{p_2}(Q)},
    \]
    where $\mathcal{C}_{\Theta,d,n}$ is the constant from \eqref{sparseConst}.
\end{thm}
\begin{proof}
    We begin by showing that there exists a sparse collection $\mathcal{S}'$, dependent on $\nu,\mu$, the input functions and fixed parameters, such that 
    \[
        |\Lambda_\mu^\nu(\vec f, \vec g)|\leq \mathcal{C}_{\Theta,d,n}\sum_{Q\in\mathcal{S}'}|Q|\llangle\vec f\rrangle_{\avgL^{p_1}(Q)}\cdot \llangle\vec g\rrangle_{\avgL^{p_2}(Q)},\numberthis\label{mainEst}
    \]
    and at the end of the proof discuss how to get rid of the dependence on $\nu,\mu$.
    
    Fix $\mu,\nu\in\Z$ and functions $\vec f\in L^{p_1}(\R^d,\C^n)$ and $\vec g\in L^{p_2}(\R^d,\C^n)$ with compact support. We may assume that $K_s=0$, if $s\notin(\mu,\nu]$. We choose the dyadic lattice $\mathscr D$ so that any compact set is contained in some dyadic cube in $\mathscr D$. Then there exists $Q_0\in\mathscr D$ such that $\supp\vec f\subset Q_0$ and $s_{Q_0}$ is bigger
    than the largest nonzero scale occurring in the kernel. Thus it suffices to show \eqref{mainEst} with $\Lambda_{Q_0,\mu,\nu}$ in place of $\Lambda_\mu^\nu$. 

    We prove by induction that, for each integer $k\geq0$,
    \begin{equation}\label{inductionStep}
    \begin{aligned}
        |\Lambda_{Q_0,\mu,\nu}(\vec f,\vec g)|\leq \mathcal{C}_{\Theta,d,n}&\sum_{l=0}^k\sum_{S\in\mathcal{S}_l'}|S|\llangle\vec f\rrangle_{\avgL^{p_1}(3S)}\cdot\llangle\vec g\rrangle_{\avgL^{p_2}(3S)}\\+&\sum_{L\in \mathcal{S}_{k+1}'}|\Lambda_{L,\mu,\nu}(\vec f,\vec g)|,
    \end{aligned}
    \end{equation}
    where each $\mathcal{S}'_l$ is a pairwise disjoint collection of dyadic subcubes of $Q_0$ such that
    \[
        \mathcal{S}'_0\coloneqq \{Q_0\},\quad \mathcal{S}_{l+1}'\coloneqq \bigcup_{Q\in\mathcal S_l'}\mathcal{S}'(Q),
    \]
    and further each $\mathcal S'(Q)$ is a pairwise disjoint collection of dyadic subcubes of $Q$ such that
    \[
        \sum_{R\in \mathcal{S}'(Q)}|R|\leq 2^{-\vartheta d} |Q|.
    \]
    The base case $k=0$ is provided by Proposition \ref{inductionBase} by setting \[\mathcal{S}_1'=\mathcal{S}'(Q_0)\coloneqq\{L\in \mathcal S_1\setcolon L\subset Q_0\},\]
    where $\mathcal{S}_1$ is the stopping collection produced by Proposition \ref{inductionBase}.

    If we assume the induction claim for some $k\geq 0$, the case $k+1$ follows immediately by applying the base step (i.e., Proposition \ref{inductionBase}), to each $L\in \mathcal S'_{k+1}$. Thus the inductive claim is verified for all $k$, and we are done with the proof of \eqref{inductionStep}.

    Next we argue that after sufficiently many iterations the second term on the right-hand side of \eqref{inductionStep} will vanish. 
    Note that each $R \in \mathcal S'(Q)$ is necessarily a strict dyadic subcube of $Q$. Hence
    each $Q \in \mathcal S'_k$ is at least $k$ generations smaller than $Q_0$. Thus if $L\in \mathcal{S}'_{k+1}$ and $k\geq s_{Q_0}-\mu$, then $s_L<s_{Q_0}-k\leq \mu$ and hence $\Lambda_{L,\mu,\nu}=0$.

    We see that from \eqref{inductionStep} it follows that 
    \[
        |\Lambda_{Q_0,\mu,\nu}(\vec f,\vec g)|\leq \mathcal{C}_{\Theta,d,n}\sum_{l=0}^k\sum_{S\in\mathcal{S}_l'}|S|\llangle\vec f\rrangle_{\avgL^{p_1}(3S)}\cdot\llangle\vec g\rrangle_{\avgL^{p_2}(3S)}
    \]
    provided $k\geq0$ is large enough. Thus we obtain \eqref{mainEst} with the sparse collection 
    \[
        \mathcal{S}'\coloneqq \bigcup_{l=0}^k\{3Q\setcolon Q\in\mathcal{S}_l'\}.
    \]
    Indeed, for any $3Q\in \mathcal{S}'$ the pairwise disjoint subsets
    \[
        E(Q)\coloneqq Q\setminus \bigcup_{R\in \mathcal{S}'(Q)}R
    \]
    satisfy $|E(Q)|\geq (1-2^{-\vartheta d})|Q|=3^{-d}(1-2^{-\vartheta d})|3Q|$.

    To get rid of the dependence on $\mu,\nu$ we note that \eqref{mainEst} implies that
    \[
        \sup_{\mu<\nu}|\Lambda_\mu^\nu(\vec f, \vec g)|\leq \mathcal{C}_{\Theta,d,n}\sup_{\mathcal{S}}\sum_{Q\in\mathcal{S}}|Q|\llangle\vec f\rrangle_{\avgL^{p_1}(Q)}\cdot \llangle\vec g\rrangle_{\avgL^{p_2}(Q)},
    \]
    where the supremum on the right-hand side is taken over all sparse collections. This implies that there exists a sparse collection $\mathcal{S}$, depending only on the input functions and fixed parameters, such that 
    \[
        \sup_{\mu<\nu}|\Lambda_\mu^\nu(\vec f, \vec g)|\leq 2\mathcal{C}_{\Theta,d,n}\sum_{Q\in\mathcal{S}}|Q|\llangle\vec f\rrangle_{\avgL^{p_1}(Q)}\cdot \llangle\vec g\rrangle_{\avgL^{p_2}(Q)},
    \]
    and the proof is done.
\end{proof}
Similarly as in the scalar case we are able to extend the result of Theorem \ref{main} to the full sesquilinear form $\Lambda\colon L^t(\R^d)\times L^{t'}(\R^d)\to\C$. The strategy of the argument is the same as in the scalar case \cite{10.2140/apde.2017.10.1255}.
Due to the limiting argument from \cite[Chapter I Paragraph 7.2]{stein1993harmonic} we have \begin{equation}\label{Steinlimitarg} \Lambda(f,g)=\langle mf,g\rangle+\lim_{\nu\to\infty}\Lambda^\nu_{-\nu}(f,g)\end{equation}
for some $m\in L^\infty(\R^d)$, whenever $f\in L^r(\R^d)$ and $g\in L^{r'}(\R^d)$. This identity easily extends to the vector valued setting. Indeed, let $\{\vec e_i\}_{i=1}^n$ be an orthonormal basis of $\C^n$ and calculate
\begin{align*}
    \Lambda(\vec f,\vec g)=\sum_{i=1}^n\Lambda(\vec f\cdot e_i,\vec g\cdot \vec e_i)
    &=\sum_{i=1}^n\left(\langle m\vec f\cdot \vec e_i,\vec g\cdot \vec e_i\rangle+\lim_{\nu\to\infty}\Lambda^\nu_{-\nu}(\vec f\cdot \vec e_i,\vec g\cdot \vec e_i)\right)\\
    &= \langle m\vec f,\vec g\rangle+\lim_{\nu\to\infty}\Lambda^\nu_{-\nu}(\vec f,\vec g).
\end{align*}
The operator $Tf\coloneqq mf$ is clearly weak type $(1,1)$ and the bi-sublinear maximal operator 
\[M_T(f,g)\coloneqq \sup_{Q\ni x}\fint_Q |T(\mathbbm 1_{(3Q)^\complement}f)||g|\numberthis\label{bisublinmaxop}\]
vanishes, and hence by \cite[Theorem 11]{muller_quantitative_2022} there exists a sparse family $\mathcal S$ of cubes such that
\[
    \langle m\vec f,\vec g\rangle\leq c_{n,d}\|m\|_{\infty}\sum_{Q\in\mathcal{S}}|Q|\llangle \vec f\rrangle_{\avgL^1(Q)}\cdot\llangle\vec g\rrangle_{\avgL^1(Q)}.
\]
These considerations show that with the assumptions of Theorem \ref{main}, for all $\vec f\in L^\infty(\R^d,\C^n)$ and $\vec g\in L^\infty(\R^d,\C^n)$ with compact support, there exists a sparse collection $\mathcal{S}$ of cubes such that
    \[
        |\Lambda(\vec f, \vec g)|\leq \mathcal{C}\sum_{Q\in\mathcal{S}}|Q|\llangle\vec f\rrangle_{\avgL^{p_1}(Q)}\cdot \llangle\vec g\rrangle_{\avgL^{p_2}(Q)},
    \]
    where $\mathcal{C}\coloneqq c_{d,n}\|m\|_{\infty}+\mathcal{C}_{\Theta,d,n}$. To extend this bound to all $\vec f\in L^{p_1}(\R^d,\C^n)$ and $\vec g\in L^{p_2}(\R^d,\C^n)$, we need the following simple lemma.
\begin{lem}\label{convBodyInc}
    Let $S_1,S_2$ be subsets of $\R^d$ and $1\leq p\leq\infty$. Then we have
    \[\llangle \mathbbm 1_{S_1}\vec f\rrangle_{\avgL^p(S_2)}\subset \llangle \vec f\rrangle_{\avgL^p(S_2)}.\]
\end{lem}
\begin{proof}
    Suppose $a\in \llangle \mathbbm 1_{S_1}\vec f\rrangle_{\avgL^p(S_2)}$. Then for some $\phi\in \bar B_{\avgL^{p'}(S_2)}$ we have
    \[
        a=\fint_{S_2}\mathbbm 1_{S_1}\vec f \,\bar \phi .
    \]
    So by defining $\psi\coloneqq \mathbbm 1_{S_1}\phi$ we get
    \[
        \|\psi\|_{\avgL^{p'}( S_2)}\leq\|\phi\|_{\avgL^{p'}(S_2)}\leq 1.
    \]
    Thus we have
    \[
        a=\fint_{S_2}\vec f \,\bar\psi\in\llangle \vec f\rrangle_{\avgL^{p}(S_2)}.
    \]
\end{proof}
Now are ready to show the following result. 
\begin{cor}\label{actualMain}
    Let $p_1,p_2,\Theta$ and $\Lambda$ be as in Theorem \ref{main}. Additionally, we assume that $\Lambda$ extends boundedly to $L^t(\R^d)\times L^{t'}(\R^d)$ for some $1<t<\infty$. Then for all $\vec f\in L^{t}(\R^d,\C^n)$ and $\vec g\in L^{t'}(\R^d,\C^n)$, there exists a sparse collection $\mathcal{S}$ of cubes such that
    \[
        |\Lambda(\vec f, \vec g)|\leq \mathcal{C}\sum_{Q\in\mathcal{S}}|Q|\llangle\vec f\rrangle_{\avgL^{p_1}(Q)}\cdot \llangle\vec g\rrangle_{\avgL^{p_2}(Q)},
    \]
    where $\mathcal{C}\coloneqq c_{d,n}\|m\|_{\infty}+\mathcal{C}_{\Theta,d,n}$ with $m\in L^\infty$ defined by \eqref{Steinlimitarg}, and $\mathcal{C}_{\Theta,d,n}$ is defined by \eqref{sparseConst}.
\end{cor}
\begin{proof}
    Let $\{Q_\ell\}_{\ell=1}^\infty$ be an increasing sequence of cubes such that $\cup_{\ell=1}^\infty Q_\ell =\R^d$, and define
    \[E_\ell\coloneqq\{x\in Q_\ell\setcolon \max\{|\vec f(x)|,|\vec g(x)|\}\leq \ell\}.\] Since $\Lambda$ extends boundedly to $L^t(\R^d)\times L^{t'}(\R^d)$, we may apply the dominated convergence theorem to get
    \begin{align*}
        |\Lambda(\vec f, \vec g)|&=\lim_{\ell\to\infty}|\Lambda(\mathbbm 1_{E_\ell}\vec f, \mathbbm 1_{E_\ell}\vec g)|.
    \end{align*}
    Then the convex body domination for compactly supported bounded functions implies that there exists a sparse collection $\mathcal{S}_\ell$ such that
    \begin{align*}
        |\Lambda(\mathbbm 1_{E_\ell}\vec f, \mathbbm 1_{E_\ell}\vec g)|
        \leq\mathcal{C}\sum_{Q\in\mathcal{S}_\ell}|Q|\llangle\mathbbm 1_{E_\ell}\vec f\rrangle_{\avgL^{p_1}(Q)}\cdot \llangle\mathbbm 1_{E_\ell}\vec g\rrangle_{\avgL^{p_2}(Q)}
    \end{align*}
    Due to Lemma \ref{convBodyInc} we have
    \[
        \llangle\mathbbm 1_{E_\ell}\vec f\rrangle_{\avgL^{p_1}(Q)}\cdot \llangle\mathbbm 1_{E_\ell}\vec g\rrangle_{\avgL^{p_2}(Q)}\leq \llangle\vec f\rrangle_{\avgL^{p_1}(Q)}\cdot \llangle\vec g\rrangle_{\avgL^{p_2}(Q)}.
    \]
    Thus it follows that
    \[|\Lambda(\vec f, \vec g)|\leq \mathcal{C}\,\sup_{\mathcal{S}}\sum_{Q\in\mathcal{S}}|Q|\llangle\vec f\rrangle_{\avgL^{p_1}(Q)}\cdot \llangle\vec g\rrangle_{\avgL^{p_2}(Q)}, \]
    which implies the wanted inequality by the same argument as in the end of the proof of Theorem \ref{main}.
\end{proof}

\section{Rough homogeneous singular integrals}
The main motivation for the abstract convex body domination principle from the last section is provided by rough singular integrals. In this section we will study the following rough homogeneous singular integrals. Let $1< q\leq \infty$. For $\Omega\in L^q(\Sphere^{d-1})$ with vanishing integral, we define
\[
    T_\Omega f(x)\coloneqq \lim_{\varepsilon\to0}\int_{|y|>\varepsilon}f(x-y)\Omega\left(\frac{y}{|y|}\right)\frac{\intD y}{|y|^d}.
\]
Our goal is to apply Theorem \ref{main} (in the form of Corollary \ref{actualMain}) to the sesquilinear form $\Lambda$ related to $T_\Omega$. Notice that the assumptions of Theorem \ref{main} are scalar valued so we may directly use the results proven in \cite{10.2140/apde.2017.10.1255} to get convex body domination for $T_\Omega$. We will include some details for clarity.

Let $\mathcal{Q}$ be a stopping collection with top $Q$. We see that if $f\in L^r(\R^d)$ and $g\in L^{r'}(\R^d)$ are compactly and disjointly supported, then the sesquilinear form $\Lambda(f,g)\coloneqq \langle T_\Omega f,g\rangle$ satisfies 
\[
    \Lambda(f,g)=\int_{\R^d}\int_{\R^d}K(x,y)f(y)\intD y\, \bar g(x)\intD x,
\]
where $K(x,y)\coloneqq \Omega\left(\frac{x-y}{|x-y|}\right)\left/\,|x-y|^d\right.$. Away from the diagonal $x=y$, the kernel $K$ admits the decomposition
\[K=\sum_sK_s,\qquad K_s(x,y)\coloneqq \Omega\left(\frac{x-y}{|x-y|}\right)2^{-sd}\phi(2^{-s}(x-y)),\numberthis\label{roughKernel}\]
 where $\phi$ is a suitable radial smooth function supported in $A_0=\{2^{-2}\leq|x|\leq1\}$. We will also need the following quasi-norm
 \[
    \|\Omega\|_{L^{q,1}\log L(\Sphere^{d-1})}\coloneqq \inf\left\{\lambda>0\setcolon \left[\Omega/\lambda\right]_{L^{q,1}\log L(\Sphere^{d-1})}\leq 1\right\}\numberthis\label{LOnormnew},
 \]
 where
 \[
    [\Omega]_{L^{q,1}\log L(\Sphere^{d-1})}\coloneqq q\int_0^\infty \log(e+t)|\{\theta\in\Sphere^{d-1}\setcolon |\Omega(\theta)|>t\}|^\frac{1}{q}\intD t\numberthis\label{LOnormold}.
 \]
Note that in \cite{10.2140/apde.2017.10.1255} the quantity \eqref{LOnormold} is denoted by $\|\Omega\|_{L^{q,1}\log L(\Sphere^{d-1})}$. We give this notation to \eqref{LOnormnew} on the account of it being homogeneous while \eqref{LOnormold} is not. 

The following result, which follows from \cite[Proposition 4.1]{10.2140/apde.2017.10.1255}, is key in verifying that the kernel $K$ satisfies the assumptions of Theorem \ref{main}.
 \begin{prop}\label{homL}
    Let $\Omega\in L^q(\Sphere^{d-1})$ with zero average and $b\in\dot {\mathcal{X}}_1(\mathcal{Q})$.
     Then there exists an absolute constant $C$ such that
     \[
        |\Lambda_{\mathcal{Q},\nu,\mu}(b,h)|\leq Cp' \tilde N_{p,q}(\Omega)\,|Q|\,\|b\|_{\dot{\mathcal{X}}_1(\mathcal{Q})}\|h\|_{\mathcal{Y}_p(\mathcal{Q})},\numberthis\label{homLineq}
     \]
     where 
     \[
        \tilde N_{p,q}(\Omega)\coloneqq \begin{cases}
            1+[\Omega]_{L^{q,1}\log L(\Sphere^{d-1})},\, &q<\infty,\, q'\leq p<\infty,\\\|\Omega\|_{L^\infty(\Sphere^{d-1})},&q=\infty,\, 1<p<\infty.
        \end{cases}
     \]
 \end{prop}
 The above result is very close to \cite[Proposition 4.1]{10.2140/apde.2017.10.1255}. The left-hand side of the inequality \eqref{homLineq} is written in a different way, but in \cite{10.2140/apde.2017.10.1255} they explain on the second paragraph after Proposition 4.1 how one can estimate $|\Lambda_{\mathcal{Q},\nu,\mu}(b,h)|$ with their Proposition 4.1 and representation \eqref{rep1}. Another difference is that in \cite[Proposition 4.1]{10.2140/apde.2017.10.1255} the function $\Omega$ has unit norm in $L^q(\Sphere^{d-1})$, which we do not assume here. This explains the different term depending on $\Omega$, which arises from the calculation below \cite[Lemma 4.3]{10.2140/apde.2017.10.1255}. With these changes the proof of Proposition \ref{homL} is identical to the proof of \cite[Proposition 4.1]{10.2140/apde.2017.10.1255}.

 We will now use Proposition \ref{homL} to show that the sesquilinear form $\Lambda(f,g)=\langle T_\Omega f,g\rangle$ satisfies the assumptions of Theorem \ref{main}.
 
\begin{lem}\label{condCheck}
    Let $\Omega$ be as in Proposition \ref{homL}. The kernel \eqref{roughKernel} satisfies conditions \eqref{supportCond}, \eqref{KernelCond} with $[K]_q\leq [\Omega]_{L^{q,1}\log L(\Sphere^{d-1})}$, \eqref{bddCond} for $r=2$ with $C_T(2)\leq C_d[\Omega]_{L^{q,1}\log L(\Sphere^{d-1})}$ and \eqref{localBddCond1}, \eqref{localBddCond2} for $(p_1,p_2)=(1,p)$, where the range of $p$ depends on $q$ in the same way as in Proposition \ref{homL}, with constant $C_L=C_dp' \tilde N_{p,q}(\Omega)$.
    
\end{lem}
\begin{proof}
    The assumptions $\eqref{supportCond}$, $\eqref{KernelCond}$ and $\eqref{bddCond}$ are quick to check for $q\in\left]1,\infty\right]$ and $r=2$. Indeed, the support condition \eqref{supportCond} follows from the support of $\phi$, and  \[[K]_q\leq \|\Omega\|_{L^q(\Sphere^{d-1})}\leq [\Omega]_{L^{q,1}\log L(\Sphere^{d-1})}\numberthis\label{KboundRSIO}\] follows immediately from definitions, so \eqref{KernelCond} is true. The boundedness condition \eqref{bddCond} for $r=2$  states that the truncations of $T_\Omega$ are $L^2$-bounded, which follows from a fundamental result \cite[Theorem 1]{Calderon1956Singular}.

 It remains to verify \eqref{localBddCond1} and \eqref{localBddCond2} with $p_1=1$ and $p_2=p$, where the range of $p$ depends on $q$ in the same way as in Proposition \ref{homL}. If $b\in\dot{\mathcal{X}}_1$, then Proposition \ref{homL} gives us 
 \[
    |\Lambda_{\mathcal{Q},\nu,\mu}(b,h)|\leq Cp'\tilde N_{p,q}(\Omega)\|b\|_{\dot{\mathcal{X}}_1(\mathcal{Q})}\|h\|_{\mathcal{Y}_p(\mathcal{Q})},
 \]
 which is \eqref{localBddCond1} with $C_L\sim p'\tilde N_{p,q}(\Omega)$. 
 
 Assume then $h\in \mathcal{Y}_{q'}$ and $b\in \dot{\mathcal{X}}_{q'}$.
 Recall that $\Lambda^*$ is like $\Lambda$, but with $\Omega^*$ in place of $\Omega$, where $\Omega^*(x)=\overline{\Omega(-x)}$. By the representation \eqref{rep2}--\eqref{rep2rem} we have 
 \[
    \Lambda_{\mathcal{Q},\nu,\mu}(h,b) =\overline{\Lambda^*_{\mathcal{Q},\nu,\mu}(b^\text{in},\mathbbm 1_Q h)}+V_\mathcal{Q}(h,b),
 \] 
where 
\[
    |V_\mathcal{Q}(h,b)|\leq 2^{\vartheta d}[K]_q|Q|\|h\|_{\mathcal{Y}_{q'}(\mathcal{Q})}\|b\|_{\dot{\mathcal{X}}_{q'}(\mathcal{Q})}
\]
for a suitable absolute constant $\vartheta$.
Now Proposition \ref{homL} with $\Omega^*\in L^q(\Sphere^{d-1})$ in place of $\Omega$ (and thus $\Lambda^*$ in place of $\Lambda$) yields
\begin{align*}
    |\Lambda_{\mathcal{Q},\nu,\mu}(h,b)|&\leq|\Lambda^*_{\mathcal{Q},\nu,\mu}(b^\text{in},\mathbbm 1_Q h)|+|V_\mathcal{Q}(h,b)|\\&\leq Cp'\tilde N_{p,q}(\Omega)\,|Q|\,\|b^{\text{in}}\|_{\dot{\mathcal{X}}_p(\mathcal{Q})}\|h\|_{\mathcal{Y}_\infty(\mathcal{Q})}\\&\qquad\qquad\qquad\qquad\qquad\qquad+2^{\vartheta d}[K]_q\,|Q|\,\|h\|_{\mathcal{Y}_{q'}(\mathcal{Q})}\|b\|_{\dot{\mathcal{X}}_{q'}(\mathcal{Q})}.
 \end{align*} 
 Then we apply $\|b^\text{in}\|_{\dot{\mathcal{X}}_p(\mathcal{Q})}\leq \|b\|_{\dot{\mathcal{X}}_p(\mathcal{Q})}$, $q'\leq p$ and \eqref{KboundRSIO} to bound the last expression by
 \begin{align*}
        (Cp'\tilde N_{p,q}(\Omega)+2^{\vartheta d}[K]_q)\,|Q|\,\|h\|_{\mathcal{Y}_\infty(\mathcal{Q})}\|b\|_{\dot{\mathcal{X}}_p(\mathcal{Q})}\lesssim_d p'\tilde N_{p,q}(\Omega)\,|Q|\,\|h\|_{\mathcal{Y}_\infty(\mathcal{Q})}\|b\|_{\dot{\mathcal{X}}_p(\mathcal{Q})}.
 \end{align*}
 Thus \eqref{localBddCond2} holds for $(p_1,p_2)=(1,p)$ with constant $C_L\sim_dp'\tilde N_{p,q}(\Omega)$.
\end{proof} 

We now see that the assumptions of Corollary \ref{actualMain} are satisfied so we get the following convex body extension of  \cite[Theorem A]{10.2140/apde.2017.10.1255}.
 \begin{thm}\label{homCBD}
     Let $\Omega\in L^1(\Sphere^{d-1})$ with zero average. Then for all $1<t<\infty$, $\vec f\in L^t(\R^d,\C^n)$ and $\vec g \in L^{t'}(\R^d,\C^n)$, there exists a sparse collection $\mathcal{S}$ of cubes such that 
     \[
            |\langle T_\Omega \vec f,\vec g\rangle|\lesssim_{d,n}p'N_{p,q}(\Omega)\sum_{Q\in\mathcal{S}}|Q|\llangle\vec f\rrangle_{\avgL^1(Q)}\cdot\llangle\vec g\rrangle_{\avgL^p(Q)},
     \]
     where 
     \[
        N_{p,q}(\Omega)\coloneqq \begin{cases}
            \|\Omega\|_{L^{q,1}\log L(\Sphere^{d-1})},\, &q<\infty,\, q'\leq p<\infty,\\\|\Omega\|_{L^\infty(\Sphere^{d-1})},&q=\infty,\, 1<p<\infty.
        \end{cases}
     \]
 \end{thm}
 \begin{proof}
     Lemma \ref{condCheck} gives that the assumptions of Corollary \ref{actualMain} are satisfied with constants bounded by $c_d\,p'\tilde N_{p,q}(\Omega)$, and thus Corollary \ref{actualMain} gives 
     \[
        |\langle T_\Omega \vec f,\vec g\rangle|\lesssim_{d,n}\left(\|m\|_{L^\infty(\R^d)}+p'\tilde N_{p,q}(\Omega)\right)\sum_{Q\in\mathcal{S}}|Q|\llangle\vec f\rrangle_{\avgL^1(Q)}\cdot\llangle\vec g\rrangle_{\avgL^p(Q)},
     \]
     where $m$ is defined by \eqref{Steinlimitarg}. Applying \eqref{Steinlimitarg} with $\|f\|_{L^2(\R^d)}=1=\|g\|_{L^2(\R^d)}$ we get
     \begin{align*}
            |\langle mf,g\rangle|&\leq  |\Lambda(f,g)|+|\lim_{\nu\to\infty}\Lambda^\nu_{-\nu}(f,g)|\\&\leq \|\Lambda\|_{L^{2}(\R^d)\times L^{2}(\R^d)\to\C}+C_{T}(2)\lesssim_d \tilde N_{p,q}(\Omega),
     \end{align*}
     and taking supremums over such $f$ and $g$ yields $\|m\|_{L^\infty(\R^d)}\lesssim_d \tilde N_{p,q}(\Omega)$. This shows that 
     \[
        |\langle T_\Omega \vec f,\vec g\rangle|\lesssim_{d,n}p'\tilde N_{p,q}(\Omega)\sum_{Q\in\mathcal{S}}|Q|\llangle\vec f\rrangle_{\avgL^1(Q)}\cdot\llangle\vec g\rrangle_{\avgL^p(Q)}.
     \]
     Suppose $q<\infty$ and $\lambda>\|\Omega\|_{L^{q,1}\log L(\Sphere^{d-1})}$. Then we have
     \begin{align*}
         |\langle T_\Omega \vec f,\vec g\rangle|&=\lambda |\langle T_{\Omega/\lambda} \vec f,\vec g\rangle|\\&\lesssim_{d,n} p'\lambda (1+[\Omega/\lambda]_{L^{q,1}\log L(\Sphere^{d-1})})\sum_{Q\in\mathcal{S}}|Q|\llangle\vec f\rrangle_{\avgL^1(Q)}\cdot\llangle\vec g\rrangle_{\avgL^p(Q)}\\&\leq 2p'\lambda \sum_{Q\in\mathcal{S}}|Q|\llangle\vec f\rrangle_{\avgL^1(Q)}\cdot\llangle\vec g\rrangle_{\avgL^p(Q)},
     \end{align*}
     and letting $\lambda\to \|\Omega\|_{L^{q,1}\log L(\Sphere^{d-1})}$ gives us the wanted result.
 \end{proof}
 We note that the above Theorem was known in the case $q=\infty$ due to \cite[Theorem 9]{muller_quantitative_2022}, but for other $q>1$ it is new. We also note that in \cite[Theorem 6.2]{AIF_2020__70_5_1871_0} an $(L^p,L^p)$ convex body bound for the maximal truncations
 \[
    T_{\Omega,\star}f(x)\coloneqq \sup_{\varepsilon>0}\left|\int_{|y|>\varepsilon}f(x-y)\Omega(\frac{y}{|y|})\frac{\intD y}{|y|^d}\right|\numberthis\label{maxTrunc}
 \]
  was achieved for the case $q=\infty$.
  
 Since the adjoint $T_\Omega^*$ is an operator of the same form as $T_\Omega$, we get $(L^p,L^1)$ bounds for $T_\Omega$ as a consequence of Theorem \ref{homCBD}.
 \begin{cor}\label{adjhomCBD}
     With the assumptions and notations of Theorem \ref{homCBD} we have
     \[
        |\langle T_\Omega \vec f,\vec g\rangle|\lesssim_{d,n}p'N_{p,q}(\Omega)\sum_{Q\in\mathcal{S}}|Q|\llangle\vec f\rrangle_{\avgL^p(Q)}\cdot\llangle\vec g\rrangle_{\avgL^1(Q)}.
     \]
 \end{cor}
 \begin{proof}
     Since the adjoint of $T_\Omega$ is $T_{\Omega^*}$, where $\Omega^*(x)\coloneqq \overline{\Omega(-x)}$, applying Theorem \ref{homCBD} to the adjoint of $T_\Omega$ implies that 
    \[
        |\langle T_\Omega \vec f,\vec g\rangle|=|\langle \vec f,T_{\Omega^*}\vec g\rangle|\lesssim_{d,n}p'N_{p,q}(\Omega^*)\sum_{Q\in\mathcal{S}}|Q|\llangle\vec f\rrangle_{\avgL^{p}(Q)}\cdot\llangle\vec g\rrangle_{\avgL^1(Q)},
    \]
    and the proof is done due to symmetry $N_{p,q}(\Omega^*)=N_{p,q}(\Omega)$.
 \end{proof}
 \section{Matrix weighted inequalities}
 The main applications of convex body domination bounds are the various matrix weighted inequalities that follow. For any $1<t<\infty$ the matrix-weighted $A_t$ class is defined by 
 \[
    [W]_{A_t}\coloneqq\fint_Q\left(\fint_Q|W^\frac{1}{t}(x)W^{-\frac{1}{t}}(y)|_{op}^{t'}\intD y\right)^\frac{t}{t'}\intD x<\infty.
 \]
 Applying Corollary \ref{adjhomCBD} with $p=q'$ and \cite[Theorem 6.9]{laukkarinen2023convex} gives us a matrix-weighted $L^t(W)\to L^t(W)$ norm inequality for $T_\Omega$ with $W\in A_{\frac{t}{q'}}$. Very recently in \cite[Theorem 1.3]{kakaroumpas2024matrixweighted} it was shown that if an operator satisfies the convex body bound of Corollary \ref{adjhomCBD}, then the condition $W\in A_\frac{t}{q'}$ is not necessary to guarantee boundedness on $L^t(W)$. In particular, their result together with Corollary \ref{adjhomCBD} lets one deduce the following conclusion.
\begin{cor}\label{weightedBounds}
    Let $1<q<\infty$, $q'<t<\infty$, $a\coloneqq\frac{tq'}{t-q'}$ and assume that \begin{equation}\label{mixedWconst}
        [W]_{A_{a',t}}\coloneqq \fint_Q\left(\fint_Q|W^\frac{1}{t}(x)W^{-\frac{1}{t}}(y)|_{op}^a\intD y\right)^\frac{t}{a}\intD x<\infty.
    \end{equation}
    Let also $\Omega\in L^{q,1}\log L(\Sphere^{d-1})$ with zero average. Then we have
    \begin{align*}
        \|T_\Omega\vec f\|_{L^t(W)}&\lesssim_{d,n,q,t}\|\Omega\|_{L^{q,1}\log L(\Sphere^{d-1})}[W]_{A_{a',t}}^{1+\frac{q'}{t(t-q')}}\|\vec f\|_{L^t(W)}\\&\leq\|\Omega\|_{L^{q,1}\log L(\Sphere^{d-1})}[W]_{A_{\frac{t}{q'}}}^{1+\frac{q'}{t(t-q')}}\|\vec f\|_{L^t(W)}.
    \end{align*}
\end{cor}
As shown in \cite[Lemma 6.6]{kakaroumpas2024matrixweighted} the second inequality in the above Corollary follows from the Cordes inequality $|A^\alpha B^\alpha|_{op}\leq |AB|_{op}^\alpha$, that holds for positive definite matrices $A,B$ and $0<\alpha<1$.

We note that \cite[Theorem 1]{muller_quantitative_2022} says that in the case $q=\infty$ we have
\[
    \|T_\Omega\vec f\|_{L^t(W)}\lesssim_{d,n,t}\|\Omega\|_{L^\infty(\Sphere^{d-1})}[W]_{A_{t}}^{1+\frac{1}{t-1}-\frac{1}{t}+\min\{1,\frac{1}{t-1}\}}\|\vec f\|_{L^t(W)}, \,1<t<\infty.\numberthis\label{matrixWineqqinfty}
\]
In \cite[Corollary 1.6]{AIF_2020__70_5_1871_0} the above bound was also formulated and proven for $t=2$ with the maximal truncated operator \eqref{maxTrunc} in place of $T_\Omega$.
\begin{rmrk}\label{cantletqtoinfinity}
The inequality \eqref{matrixWineqqinfty} does not follow from Corollary \ref{adjhomCBD} and \cite[Theorem 6.9]{laukkarinen2023convex} (or \cite[Theorem 1.3]{kakaroumpas2024matrixweighted}) since we do not have Corollary \ref{adjhomCBD} for $p=1=q'$.
\end{rmrk}

Given a locally integrable function $b$ we define the commutator $[b,T]$ of a linear operator $T$ by $[b,T]f\coloneqq T(bf)-bTf$.
By Corollary \ref{adjhomCBD} and \cite[Theorem 7.4]{laukkarinen2023convex} we also get matrix-weighted bounds for the commutator of $T_\Omega$. We may also refine this bound by combining the techniques in the proofs of \cite[Theorem 1.3]{kakaroumpas2024matrixweighted} and \cite[Theorem 7.4]{laukkarinen2023convex}. This lets us replace the weight constant $[W]_{\frac{t}{q'}}$ with \eqref{mixedWconst}.
\begin{cor}\label{commWeigthBounds}
    Let $1<q<\infty$, $q'<t<\infty$, $a\coloneqq\frac{tq'}{t-q'}$ and $b\in BMO$. Let also $\Omega\in L^{q,1}\log L(\Sphere^{d-1})$ with zero average. Then we have
    \begin{align*}
        \|[b,T_\Omega]\vec f \|_{L^t(W)}&\lesssim_{d,n,q,t}\|b\|_{BMO}\|\Omega\|_{L^{q,1}\log L(\Sphere^{d-1})}[W]_{A_{a',t}}^{1+\frac{q'}{t(t-q')}+\max\{1,\frac{q'}{t-q'}\}}\|\vec f\|_{L^t(W)}\\
        &\leq\|b\|_{BMO}\|\Omega\|_{L^{q,1}\log L(\Sphere^{d-1})}[W]_{A_\frac{t}{q'}}^{1+\frac{q'}{t(t-q')}+\max\{1,\frac{q'}{t-q'}\}}\|\vec f\|_{L^t(W)},
    \end{align*}
    where $[W]_{a',t}$ is as in \eqref{mixedWconst}.
\end{cor}
We also record and prove a new matrix-weighted estimate for the commutators in the case that $\Omega\in L^\infty(\Sphere^{d-1})$. Due to the same reason as in Remark \ref{cantletqtoinfinity}, we are not allowed to simply let $q\to\infty$ in Corollary \ref{commWeigthBounds}. 
Note that we could directly apply Theorem \ref{homCBD} (or Corollary \ref{adjhomCBD}) and \cite[Theorem 7.4]{laukkarinen2023convex} to deduce matrix weighted bounds for the commutators in this case as well. However, here we have the liberty of choosing the parameter $p$ in the convex body domination, which means that we can achieve more precise bounds than in \cite{laukkarinen2023convex}.
One could revisit to the proofs of \cite[Theorem 7.4]{laukkarinen2023convex} and \cite[Theorem 1] {muller_quantitative_2022} to prove the following result, but we will present a shorter proof that uses matrix-weighted bounds for commutators from \cite{isralowitz_commutators_2022}.
  
\begin{cor}\label{cominftycor}
    Let $\Omega\in L^\infty(\Sphere^{d-1})$ with zero average and $b\in BMO$. Then 
    \[
        \|[b,T_\Omega]\vec f\|_{L^t(W)}\lesssim_{d,n,t}\|b\|_{BMO}\|\Omega\|_{L^\infty(\Sphere^{d-1})}[W]_{A_t}^{2+\frac{2}{t-1}-\frac{1}{t}} \|\vec f\|_{L^t(W)}.
    \]
\end{cor}
\begin{proof}
    As discussed before, \cite[Theorem 1]{muller_quantitative_2022} says that for $1<t<\infty$ we have
    \[
        \|T_\Omega\|_{L^t(W)\to L^t(W)}\lesssim_{d,n,t}\phi([W]_{A_t}),
    \]
    where $\phi(x)\coloneqq \|\Omega\|_{L^\infty(\Sphere^{d-1})}x^{1+\frac{1}{t-1}-\frac{1}{t}+\min\{1,\frac{1}{t-1}\}}$. Thus due to \cite[Proposition 1.5]{isralowitz_commutators_2022} there exists a constant $C\coloneqq C(d,n,t)$ such that
    \begin{align*}
         \|[b,T_\Omega]\|_{L^t(W)\to L^t(W)}&\lesssim_{d,n,t}\|b\|_{BMO}([W]_{A_t}+[W^{-\frac{t'}{t}}]_{A_{t'}})\,\phi(C[W]_{A_t})\\&\lesssim_{n,t}\|b\|_{BMO}[W]_{A_t}^{\max\{1,\frac{1}{t-1}\}}\phi(C[W]_{A_t})\\&\eqsim_{d,n,t}\|b\|_{BMO}\|\Omega\|_{L^\infty(\Sphere^{d-1})}[W]_{A_t}^{2+\frac{2}{t-1}-\frac{1}{t}},
    \end{align*}
    and the proof is done.
\end{proof}
\section{Sparse domination of commutators and Bloom-type bounds}
Another interesting result that follows from Theorem \ref{homCBD} is a certain sparse domination bound for the commutators. To formulate this result we need the following quantity.
For a collection of dyadic cubes $\mathscr Q$, functions $\vec h_1,\vec h_2\in L^1_{loc}(\R^d,\C^n)$, $b\in L^1_{loc}(\R^d,\C)$ and positive numbers $\gamma,\beta\geq 1$  we define 
\[
        \mathcal{A}^{\gamma,\beta}_{b,\mathscr Q}(\vec h_1,\vec h_2)\coloneqq \sum_{Q\in\mathscr Q}|Q|\left(\fint_Q |b(x)-\langle b\rangle_Q|^{\gamma}|\vec h_1(y)|^{\gamma}\intD y\right)^\frac{1}{\gamma}\left(\fint_Q |\vec h_2(x)|^{\beta}\intD x\right)^\frac{1}{\beta},\numberthis\label{bilinearcommutatorsparseform}
    \]
    where $\langle b\rangle_Q$ is the average of $b$ over $Q$. In \cite[Proposition 7.1]{laukkarinen2023convex} it was shown that convex body domination implies  that the bilinear sparse forms \eqref{bilinearcommutatorsparseform} can be used to dominate the commutators. The following result is an immediate consequence of Theorem 5.3 and Corollary 5.4 combined with  \cite[Proposition 7.1]{laukkarinen2023convex}.
\begin{cor}\label{CommSparse}
    Let $\Omega\in L^1(\Sphere^{d-1})$ with zero average and $b\in L^1_{loc}(\R^d,\C)$. Then for all $1<t<\infty$, $\vec f\in L^t(\R^d,\C^n)$ and $\vec g \in L^{t'}(\R^d,\C^n)$, there exists a sparse collection $\mathcal{S}$ of cubes such that 
    \[
        |\langle [b,T_\Omega]\vec f, \vec g\rangle|\lesssim_{d,n}p'N_{p,q}(\Omega)\,[\mathcal{A}^{\gamma,\beta}_{b,\mathcal S}(\vec f,\vec g)+\mathcal{A}^{\beta,\gamma}_{b,\mathcal S}(\vec g,\vec f)],
    \]
    where $(\gamma,\beta)\in\{(1,p),(p,1)\}$ and $N_{p,q}(\Omega)$ is as in Theorem \ref{homCBD}.
\end{cor}
The above result was known for $q=\infty$ due to \cite{Lerner2024bloom}, but it is completely new for $q<\infty$. In \cite{Lerner2024bloom}, it was shown that this type of domination lets us deduce so called Bloom-type boundedness for $T_\Omega$. In order to introduce this result we need some definitions. The reverse Hölder class $RH_s$ is a class of scalar weights $w$ that satisfy
\[
    [w]_{RH_s}\coloneqq\sup_Q\left(\fint_Qw^s\right)^\frac{1}{s}\left(\fint_Q w\right)^{-1}<\infty.
\]
The weighted fractional BMO-seminorm is defined by
\[
    \|b\|_{BMO^\alpha_\nu}\coloneqq\sup_{Q}\frac{1}{\nu(Q)^{1+\frac{\alpha}{d}}}\int_Q|b-\langle b\rangle_Q|,
\]
where $\alpha\geq0$, $\nu$ is a weight and $\nu(Q)\coloneqq\int_Q\nu$. Furthermore,
we need the weighted sharp maximal function
\[
    M_\nu^\#b\coloneqq \sup_{Q\ni x}\frac{1}{\nu(Q)}\int_Q|b-\langle b\rangle_Q|.
\]
By \cite[Theorem 5.1]{Lerner2024bloom} the bilinear sparse forms \eqref{bilinearcommutatorsparseform} can be estimated in $L^u(\mu)\times L^v(\lambda^{1-q'})$ with suitable parameters $u,v$ and weights $\mu,\lambda$.
The following Corollary is an immediate consequence of Corollary \ref{CommSparse} and \cite[Theorem 5.1]{Lerner2024bloom}.
\begin{cor}\label{bloomBound}
    Let $\Omega\in L^{q,1}\log L(\Sphere^{d-1})$ for some $1<q<\infty$ with zero average. Let also $1<u,v<q$, $\mu\in A_u\cap RH_{\left(\frac{q}{u}\right)'}$ and $\lambda\in A_v\cap RH_{\left(\frac{q}{v}\right)'}$, and define the Bloom-weight $\nu^{1+\alpha}\coloneqq \mu^\frac{1}{u}\lambda^{-\frac{1}{v}}$, where $\alpha\coloneqq-\frac{1}{\tau}\coloneqq\frac{1}{u}-\frac{1}{v}$. Then we have
    \begin{align*}
        \|[b,T_\Omega]\|_{L^u(\mu)\to L^v(\lambda)}\lesssim_{d,q,u,v,\mu,\lambda} \|\Omega\|_{L^{q,1}\log L(\Sphere^{d-1})}\, 
        \mathcal{N}_{u,v}(b)
    \end{align*}
    where 
    \[
        \mathcal{N}_{u,v}(b)\coloneqq\begin{cases}
            \|b\|_{BMO^{\alpha d}_\nu}, \quad\text{ if } u\leq v,\\
            \|M^\#_\nu b\|_{L^\tau(\nu)},\quad \text{ if } u\geq v.
        \end{cases}
    \] 
\end{cor}
The proof of above result is essentially the same as the proof of \cite[Theorem 6.1]{Lerner2024bloom} with the simplification $r=1$. The difference being that in \cite{Lerner2024bloom} they assume $\Omega\in L^\infty$ and the sparse domination of commutators follows from the local weak-type control of the grand maximal truncation operator. This type of control is not known for unbounded $\Omega$ so in our case we need Corollary \ref{CommSparse} for the sparse domination. An explicit dependence on $\mu$ and $\lambda$ can be easily tracked from the proof of \cite[Theorem 6.1]{Lerner2024bloom}.

 \section{Bochner-Riesz means at the critical index}
 Another operator that can be estimated via Theorem \ref{main} is given by the critical index Bochner-Riesz mean, i.e., the Fourier multiplier operator 
 \[
    B_\delta f\coloneqq \mathcal{F}^{-1}\left[\widehat f (\cdot)(1-|\cdot|^2)^\delta_+\right],\qquad \delta\coloneqq \frac{d-1}{2}.
 \]
 Same techniques as in \cite{10.2140/apde.2017.10.1255} could also be applied to show that the Bochner-Riesz means can be estimated using Theorem \ref{main}. However, thanks to recent results of Shrivastava and Shuin \cite{shuin2019Composition,shuin2019Corrigendum}, there is now the following shorter route to such a result via the grand maximal truncation operator approach due to Muller and Rivera-R\'ios \cite{muller_quantitative_2022}.

 The following result was originally proven implicitly in \cite{shuin2019Composition,shuin2019Corrigendum}. See also \cite[Example 6.6]{Lerner2024bloom}.
\begin{thm}\label{BRMgrandweakboundedness}
    Let $\delta\coloneqq \frac{d-1}{2}$ and $1\leq s<\infty$. The grand maximal truncation operator defined by 
    \[
        M_{B_\delta,s} f(x)\coloneqq\sup_{Q\ni x}\left(\fint_Q|B_\delta(f\mathbbm 1_{(3Q)^\complement})|^s\right)^\frac{1}{s}
    \]
    satisfies
    \[
        \|M_{B_\delta,s}\|_{L^{1,\infty}\to L^1}\lesssim_d s.
    \]
\end{thm}
The classical result of Christ \cite{Christ1998Weak} and the above Theorem combined with convex body domination principle of Muller and Rivera-R\'ios \cite{muller_quantitative_2022} gives us convex body domination for the Bochner-Riesz means at the critical index.
\begin{thm}\label{CBDBRmeans}
    Let $1<p<\infty$ and $\delta=\frac{d-1}{2}$. For all $1<t<\infty$, $\vec f\in L^t(\R^d,\C^d)$ and $\vec g\in L^{t'}(\R^d,\C^d)$, there exists a sparse collection $\mathcal{S}$ of cubes such that, we have
    \[
        |\langle B_\delta \vec f,\vec g\rangle|\lesssim_{d,n}p' \sum_{Q\in\mathcal{S}}|Q|\,\llangle\vec f\rrangle_{\avgL^1(Q)}\cdot\llangle\vec g\rrangle_{\avgL^p(Q)}.
    \]
\end{thm}
\begin{proof}
    We denote $\tilde p\coloneqq\frac{p}{p+1}$. By \cite[Theorem 11]{muller_quantitative_2022} we have 
    \[
        |\langle B_\delta f,g\rangle|\lesssim_{d,n} (\|B_\delta\|_{L^1\to L^{1,\infty}}+\|M_{B_\delta}\|_{L^{1}\times L^p\to L^{\tilde p,\infty}})\sum_{Q\in\mathcal{S}}|Q|\,\llangle\vec f\rrangle_{\avgL^1(Q)}\cdot\llangle\vec g\rrangle_{\avgL^p(Q)},
    \]
    where $M_{B_\delta}$ is the bilinear grand maximal truncation operator \eqref{bisublinmaxop} with $T=B_\delta$. The main result of \cite{Christ1998Weak} is that $\|B_\delta\|_{L^1\to L^{1,\infty}}\lesssim_d 1$ so it remains to check that
    \[
        \|M_{B_\delta}\|_{L^{1}\times L^p\to L^{\tilde p,\infty}}\lesssim_d p'.
    \]
    By Hölder's inequality we have 
    \begin{align*}
        M_{B_\delta}(f,g)=\sup_{Q\ni x}\fint_Q |B_\delta(\mathbbm 1_{(3Q)^\complement}f)||g|\leq M_{B_\delta,p'}fM_pg.
    \end{align*}
    Then the Hölder inequality for weak $L^p$ spaces gives
    \[
        \|M_{B_\delta}(f,g)\|_{L^{\tilde p,\infty}}\lesssim \|M_{B_\delta,p'}f\|_{L^{1,\infty}}\|M_pg\|_{L^{p,\infty}}.
    \]
    Lastly, by Theorem \ref{BRMgrandweakboundedness} with $s=p'$ and the weak-type boundedness of the maximal operator we have 
    \[
        \|M_{B_\delta,p'}f\|_{L^{1,\infty}}\|M_pg\|_{L^{p,\infty}}\lesssim_{d} p'\|f\|_{L^1}\|g\|_{L^p},
    \]
    and the proof is done.
\end{proof}
With the same proof as in Corollary \ref{adjhomCBD} the $(L^1,L^p)$ bound for $B_\delta$ implies an $(L^p,L^1)$ bound. Again, as a Corollary we get the following matrix-weighted norm inequality.
\begin{cor}
    For $1<t<\infty$ we have
    \[
        \|B_\delta\vec f\|_{L^t(W)}\lesssim_{d,n,t}[W]_{A_{t}}^{1+\frac{1}{t-1}-\frac{1}{t}+\min\{1,\frac{1}{t-1}\}}\|\vec f\|_{L^t(W)}.
    \]
\end{cor}
The proof of the above result is essentially contained in the proof of Theorem 1 from \cite{muller_quantitative_2022}. 
Lastly, with the same proof as in Corollary \ref{cominftycor} we see that the commutators satisfy the following matrix-weighted inequality.
\begin{cor}
    Let $b\in BMO$. Then 
    \[
        \|[b,B_\delta]\vec f\|_{L^t(W)}\lesssim_{d,n,t}\|b\|_{BMO}[W]_{A_t}^{2+\frac{2}{t-1}-\frac{1}{t}} \|\vec f\|_{L^t(W)}.
    \]
\end{cor}
We can also get analogous results to Corollaries \ref{CommSparse} and \ref{bloomBound} for $B_\delta$, but these results were already known due to \cite{Lerner2024bloom}, so we will not repeat them here.

\bibliographystyle{plain}
\bibliography{RSIrefs}

\begin{thebibliography}{10}

\bibitem{Calderon1956Singular}
A.~P. Calderon and A.~Zygmund.
\newblock On singular integrals.
\newblock {\em American Journal of Mathematics}, 78(2):289--309, 1956.

\bibitem{Christ1998Weak}
M.~Christ.
\newblock Weak type (1, 1) bounds for rough operators.
\newblock {\em Annals of Mathematics}, 128(1):19--42, 1988.

\bibitem{10.2140/apde.2017.10.1255}
J.~M. Conde-Alonso, A.~Culiuc, F.~Di Plinio, and Y.~Ou.
\newblock {A sparse domination principle for rough singular integrals}.
\newblock {\em Analysis \& PDE}, 10(5):1255--1284, 2017.

\bibitem{AIF_2020__70_5_1871_0}
F.~Di~Plinio, T.~P. Hyt\"onen, and K.~Li.
\newblock Sparse bounds for maximal rough singular integrals via the {Fourier} transform.
\newblock {\em Annales de l'Institut Fourier}, 70(5):1871--1902, 2020.

\bibitem{DOMELEVO2024127956}
K.~Domelevo, S.~Kakaroumpas, S.~Petermichl, and O.~{Soler i Gibert}.
\newblock Boundedness of {J}ourné operators with matrix weights.
\newblock {\em Journal of Mathematical Analysis and Applications}, 532(2):127956, 2024.

\bibitem{Domelevo_Matrix_A2_Conj}
K.~Domelevo, S.~Petermichl, S.~Treil, and A.~Volberg.
\newblock The matrix ${A_2}$ conjecture fails, i.e. $3/2 > 1$, 2024.
\newblock Preprint, arXiv:2402.06961.

\bibitem{duong_variation_2021}
X.~T. Duong, J.~Li, and D.~Yang.
\newblock Variation of {Calderón}–{Zygmund} operators with matrix weight.
\newblock {\em Communications in Contemporary Mathematics}, 23(07):2050062, 2021.

\bibitem{HYTONEN2024127565}
T.~P. Hytönen.
\newblock Some remarks on convex body domination.
\newblock {\em Journal of Mathematical Analysis and Applications}, 529(2):127565, 2024.

\bibitem{isralowitz_sharp_2021}
J.~Isralowitz, S.~Pott, and I.~P. Rivera-R\'ios.
\newblock Sharp ${A}_1$ weighted estimates for vector-valued operators.
\newblock {\em The Journal of Geometric Analysis}, 31(3):3085--3116, 2021.

\bibitem{isralowitz_commutators_2022}
J.~Isralowitz, S.~Pott, and S.~Treil.
\newblock Commutators in the two scalar and matrix weighted setting.
\newblock {\em Journal of the London Mathematical Society}, 106(1):1--26, 2022.

\bibitem{kakaroumpas2024matrixweighted}
S.~Kakaroumpas, T.H. Nguyen, and D.~Vardakis.
\newblock Matrix-weighted estimates beyond {C}alder\'on-{Z}ygmund theory, 2024.
\newblock Preprint, arXiv:2404.02246.

\bibitem{kakaroumpas2024matrixweightedold}
S.~Kakaroumpas, T.H. Nguyen, and D.~Vardakis.
\newblock Matrix-weighted estimates beyond {C}alder\'on-{Z}ygmund theory, 2024.
\newblock Preprint, arXiv:2404.02246v1.

\bibitem{laukkarinen2023convex}
A.~Laukkarinen.
\newblock Convex body domination for a class of multi-scale operators, 2023.
\newblock Preprint, arXiv:2311.10442.

\bibitem{laukkarinencompactness}
A.~Laukkarinen and J.~Sinko.
\newblock Compactness of commutators of rough singular integrals.
\newblock Manuscript in preparation.

\bibitem{Lerner2024bloom}
A.~K. Lerner, E.~Lorist, and S.~Ombrosi.
\newblock Bloom weighted bounds for sparse forms associated to commutators.
\newblock {\em Mathematische Zeitschrift}, 306(4):73, 2024.

\bibitem{muller_quantitative_2022}
P.~A. Muller and I.~P. Rivera-R\'ios.
\newblock Quantitative matrix weighted estimates for certain singular integral operators.
\newblock {\em Journal of Mathematical Analysis and Applications}, 509(1):125939, 2022.

\bibitem{nazarov_convex_2017}
F.~Nazarov, S.~Petermichl, S.~Treil, and A.~Volberg.
\newblock Convex body domination and weighted estimates with matrix weights.
\newblock {\em Advances in Mathematics}, 318:279--306, 2017.

\bibitem{shuin2019Composition}
K.~Shuin and S.~Shrivastava.
\newblock On composition of maximal function and {B}ochner-{R}iesz operator at the critical index.
\newblock {\em Proceedings of the American Mathematical Society}, 148(4):1545--1554, 2019.

\bibitem{shuin2019Corrigendum}
K.~Shuin and S.~Shrivastava.
\newblock Corrigendum to “{O}n composition of maximal function and {B}ochner-{R}iesz operator at the critical index".
\newblock {\em Proceedings of the American Mathematical Society}, 149(7):3139--3141, 2021.

\bibitem{stein1993harmonic}
E.~M. Stein.
\newblock {\em Harmonic Analysis: Real-variable Methods, Orthogonality, and Oscillatory Integrals}.
\newblock Monographs in harmonic analysis. Princeton University Press, 1993.

\end{thebibliography}
\end{document}